\documentclass{art-en}

\usepackage{multicol}

\allowdisplaybreaks

\DeclareMathOperator\Hom{Hom}

\DeclareMathOperator\kera{ker}
\DeclareMathOperator\cokera{coker}

\DeclareMathOperator\diagMat{diag}
\DeclareMathOperator\Mat{Mat}
\DeclareMathOperator\GL{GL}
\DeclareMathOperator\Otn{O}
\DeclareMathOperator\Id{Id}

\newcommand\Q{\mathbb{Q}}

\DeclareMathOperator\Sym{Sym}
\DeclareMathOperator\K{K}
\DeclareMathOperator\GW{GW}
\DeclareMathOperator\W{W}

\newcommand\sym{\mathrm{s}}
\newcommand\qud{\mathrm{q}}

\newcommand\m{\mathfrak{m}}
\newcommand\N{\mathbb{N}}
\newcommand\Z{\mathbb{Z}}
\newcommand\F{\mathbb{F}}
\newcommand\perm{\mathfrak{S}}

\newcommand\ff[1]{\F_{#1}}
\newcommand\cyc[1]{{\Z/{#1}\Z}}

\DeclarePairedDelimiter\smb{\langle}{\rangle}
\DeclarePairedDelimiter\pfister{\langle\!\langle}{\rangle\!\rangle}
\DeclarePairedDelimiterX\scal[2]{\langle}{\rangle}{#1, #2}
\DeclarePairedDelimiterX\comprehension[2]{\{}{\}}{\; #1 \;\delimsize\vert\; #2 \;}

\title[\textbf{A presentation of the Grothendieck--Witt group over \texorpdfstring{$\ff{2}$}{F_2}}]{\textbf{A presentation of the symmetric Grothendieck--Witt group of local rings over \texorpdfstring{$\ff{2}$}{F_2}}}
\author{Marcus Nicolas}

\begin{document}

  \begin{abstract}
    Let $R$ be a commutative local ring.
    We provide an explicit presentation of the symmetric Grothendieck--Witt
    ring $\GW^\sym(R)$ of $R$ as an abelian group when $R$
    has residue field $\ff{2}$.
    This completes \cite{RS24}, where an explicit presentation of $\GW^\sym(R)$
    is given when the residue field is different from $\ff{2}$.
    We then use this result to compute the symmetric Grothendieck--Witt rings
    for the sequences of local rings $\cyc{2^n}$ and $\ff{2}[x] / (x^n)$.
  \end{abstract}

  \maketitle

  \tableofcontents

  \section*{Introduction}

The study of symmetric and quadratic bilinear forms over commutative
rings occupies a central role in algebra, geometry and number theory.
However, their direct classification is often highly intricate,
which motivates the introduction of more computable invariants such as the
Grothendieck--Witt $\GW(R)$ and Witt rings $\W(R)$ attached to a commutative
ring $R$, strongly related to its algebraic $\K$-theory $\K(R)$.

The symmetric Witt ring $\W^\sym(k)$ of a field $k$ admits a simple
presentation \cite{MH73}*{lemma IV.1.1} : it is additively generated
by symbols $\smb{a}$ for $a \in k^\times / k^{\times 2}$ and relations
\begin{enumerate}
  \item for $a$ unit :
  \[
    \smb{a} + \smb{-a} = 0
  \]

  \item for $a$ and $b$ two units such that $a+b$ is again a unit :
  \[
    \smb{a} + \smb{b} = \smb{a+b} + \smb{a b (a+b)}
  \]
\end{enumerate}
In their recent paper \cite{RS24}, Rogers and Schlichting compute the
symmetric Grothendieck--Witt ring of a local ring $R$ of residue field
not $\F_2$, and obtain as a corollary that the same presentation holds
in this case for $\W^\sym(R)$.

However, observe that this presentation is not correct in general when
$R$ has residue field $\ff{2}$. Indeed, since in this case the sum of two
units cannot be a unit itself, the second family of relations is always empty.
Recall that the ring $\W^\sym(\Z_2)$ is contained
in $\W^\sym(\Q_2)$ \cite{MH73}*{corollary IV.3.3},
and is thus finite \cite{Lam05}*{theorem VI.2.29}.
Since the only units of $\cyc{4}$ are $1$ and $3$, the above presentation would
in this case imply $\W^\sym(\cyc{4}) \otimes \Q \iso \Q$,
which is not an algebra over $\W^\sym(\Z_2) \otimes \Q \iso 0$.
This presentation is also incomplete for $\F_2[x] / (x^4)$,
see \cite{RS24}*{proposition 4.1}.

The main goal of this paper is to give explicit presentations of
$\GW^\sym(R)$ and $\W^\sym(R)$ for a commutative local ring $R$ whose
residue field is $\F_2$.

Our main tool will be the following general and rather inexplicit presentation
established in \cite{KRW72} :
if $R$ is a commutative semi-local connected
ring $R$, then its symmetric Grothendieck--Witt ring $\GW^\sym(R)$
is additively generated by symbols $\smb{a}$ for
$a \in R^\times / R^{\times 2} $, subject to the relations
\[
  \smb{a_1} + \cdots + \smb{a_4} = \smb{b_1} + \cdots + \smb{b_4}
\]
for $a_1, \dots, b_4$ such that $(a_1) \perp \cdots \perp (a_4) \iso
(b_1) \perp \cdots \perp (b_4)$.
The exact statement is recalled below as \cref{thm[1].presentation}.
Here $(x)$ for a unit $x$ denotes the module $R$ endowed with the unimodular
symmetric bilinear form $u \otimes v \mapsto x u v$, and $\perp$ is the
orthogonal sum.

\vspace{\baselineskip}

The main result is the following :

\begin{thm*}[stated at \ref{thm[3].presentation}]
  Let $R$ be a commutative local ring with residue field $\ff{2}$.
  Then $\GW^\sym(R)$ is generated as an abelian group by
  symbols $\smb{a}$ for $a \in R^\times / R^{\times 2}$ and relations :
  \begin{enumerate}[label=(\arabic*)]
    \item\label{item[0].even_relation} for units $a$, $b$
    and $m$, $n$ non-invertible such that $ma + nb = 0$ :
    \[
      \smb{a} + \smb{b} = \smb{a + n^2 b} + \smb{b + m^2 a}
    \]

    \item\label{item[0].odd_relation} for units $a$, $b$, $c$ and $d$ :
    \[
      \smb{a} + \smb{b} + \smb{c} + \smb{d} =
        \smb{u} + \smb{v} + \smb{w} + \smb{a b c d \cdot u v w}
    \]
    where $u \defeq a + \frac{1}{c} + \frac{1}{d}$, $v \defeq \frac{1}{a} +
    \frac{1}{b} + d$ and $w \defeq a + b + a^2 c$.
  \end{enumerate}
\end{thm*}

\vspace{\baselineskip}

In the last section, we show how this result can be used to compute
the symmetric Grothendieck--Witt groups of the rings
$\cyc{2^n}$ and $\ff{2}[x] / (x^n)$ for $n \geq 2$,
see \cref{prop[4].calcul_cyc4_cyc8,cor[4].tower_cyc2n,cor[4].calcul_trunc}.
More precisely, we obtain explicit group isomorphisms
\[
  \GW^\sym(\cyc{2^n}) \iso
    \begin{cases} 
      \Z \oplus \cyc{4} & \text{if } n = 2 \\
      \Z \oplus \cyc{4} \oplus \cyc{2} & \text{if } n \geq 3
    \end{cases}
\]
and
\[
  \GW^\sym\!\big(\ff{2}[x] / (x^n) \big) \iso
    \begin{cases} 
      \Z \oplus (\cyc{2})^k & \text{if } n = 2k \\
      \Z \oplus (\cyc{2})^k & \text{if } n = 2k+1
    \end{cases}
\]
We furthermore prove that the canonical comparison maps
\[
  \GW^\sym(\Z_2)
    \rightarrow \lim_n \GW^\sym(\cyc{2^n})
  \quad\text{and}\quad
  \GW^\sym(\ff{2}\llbracket x \rrbracket)
    \rightarrow \lim_n \GW^\sym\!\big(\ff{2}[x] / (x^n) \big)
\]
are isomorphisms, see \cref{cor[4].tower_cyc2n,cor[4].comparison_trunc}.
In all of these cases, we give an explicit description
of the ring structure on $\GW^\sym$.

  \section{Setting the stage}

We first recall quickly the definition of the symmetric Grothendieck--Witt
and Witt rings associated to a commutative ring $R$.

\begin{dfn}
  Let $R$ be a commutative ring.
  \begin{enumerate}[label=(\arabic*)]
    \item A \nterm{symmetric $R$-module} is a pair $(M, \phi)$ where $M$ is
    a $R$-module and $\phi \in \Hom_{R}(M \otimes_R M, R)^{\cyc{2}}$
    is a symmetric form on $M$. We say that $(M, \phi)$ is \nterm{unimodular} if
    $\phi$ is. A morphism between symmetric $R$-modules is an $R$-linear
    morphism intertwining the symmetric forms.

    \item Given a unimodular symmetric $R$-module $(M, \phi)$,
    a \nterm{lagrangian} is an inclusion $(L, 0) \subseteq (M, \phi)$ such
    that the induced map $M / L \rightarrow L^\vee$ is an isomorphism.
    A symmetric module that has a lagrangian is said to be \nterm{metabolic}.

    \item Given two symmetric $R$-modules $(M, \phi)$ and $(N, \psi)$,
    their \nterm{orthogonal sum} is the symmetric $R$-module
    $(M, \phi) \perp (N, \psi) \defiso (M \oplus N, \phi \perp \psi)$.
  \end{enumerate}
\end{dfn}

\begin{discussion}
  Given a commutative ring $R$, the orthogonal sum evidently equips the
  set $\Sym(R)$ of isomorphism classes of finitely generated projective
  symmetric unimodular $R$-modules with the structure of a commutative monoid.
\end{discussion}

\begin{dfn}
  If $R$ is a commutative ring, define its \nterm{symmetric Grothendieck--Witt
  ring} $\GW^\sym(R)$ as the group completion of the commutative monoid
  $\Sym(R)$. Observe that the tensor product of symmetric modules promotes
  $\GW^\sym(R)$ to a commutative ring.

  The \nterm{symmetric Witt ring} $\W^\sym(R)$ of $R$ is then defined as the
  quotient of $\GW^\sym(R)$ by the ideal generated by metabolic forms under
  the canonical map $\Sym(R) \rightarrow \GW^\sym(R)$.
\end{dfn}

\begin{cstr}
  Let $R$ be a commutative ring.
  The kernel of the group morphism
  $R^\times \rightarrow \Sym(R)^\times$
  contains $R^{\times 2}$, and the composite
  \[
    R^\times / R^{\times 2}
      \rightarrow \Sym(R)^\times
      \rightarrow \GW^\sym(R)^\times
  \]
  thus induces a canonical ring morphism
  \[
    \Z[R^\times / R^{\times 2}] \rightarrow \GW^\sym(R)
  \]
  by the adjunction $\Z[-] \dashv (-)^\times$.

  We denote $a \mapsto \smb{a}$ the canonical projection
  $R^\times \rightarrow R^\times / R^{\times 2}$.
\end{cstr}

When $R$ is local, or more generally semi-local and connected, its
symmetric Grothendieck--Witt group is in fact completely
determined by the behaviour of symmetric forms on $R^4$ :

\begin{thm}[\cite{KRW72}*{theorem 1.16}]\label{thm[1].presentation}
  If $R$ is a commutative semi-local connected ring, then the ring morphism
  \[
    \Z[R^\times / R^{\times 2}] \rightarrow \GW^\sym(R)
  \]
  is surjective, and its kernel is additively generated by the relations
  \[
    \smb{a_1} + \cdots + \smb{a_4} = \smb{b_1} + \cdots + \smb{b_4}
  \]
  for units $a_1, \dots, b_4$ such that $(a_1) \perp \cdots \perp (a_4) \iso
  (b_1) \perp \cdots \perp (b_4)$.
\end{thm}

\begin{rmk}\label{rmk[1].base_change_otn}
  Let $(R, \m)$ be a local ring, and $a_1, \dots, b_4$ units.
  The statement
  \[(a_1) \perp \cdots \perp (a_4) \iso (b_1) \perp \cdots \perp (b_4)\]
  is equivalent to the existence of $P \in \GL_4(R)$ such that
  \[
    P \diagMat(a_1, \dots, a_4) \, P^\top = \diagMat(b_1, \dots, b_4)
  \]
  In particular if $R / \m \iso \ff{2}$, then
  $P P^\top \equiv \Id_4 \pmod{\m}$
  and $P$ lies over an \emph{orthogonal} matrix of $\GL_4(\ff{2})$.
\end{rmk}

We conclude this section with a description of the orthogonal groups of $\ff{2}$
in low dimensions, motivated by the previous remark.

\begin{dfn}
  For $n \geq 1$ and $R$ a commutative ring, define
  the \emph{orthogonal group} of $R^n$ by
  \[
    \Otn_n(R) \defeq \comprehension{P \in \GL_n(R)}{PP^\top = \Id_n}
  \]
\end{dfn}

\begin{prop}\label{prop[1].otn_cyc2}
  The first orthogonal groups of $\ff{2}$ are given by
  \[
    \Otn_2(\ff{2}) \iso \perm_2
    \text{, }\quad
    \Otn_3(\ff{2}) \iso \perm_3
    \quad\text{and}\quad
    \Otn_4(\ff{2}) \iso \perm_4 \times \cyc{2}
  \]
  where in the last isomorphism $\perm_4$ and $1 \in \cyc{2}$ are respectively
  identified with the subgroup of permutation matrices and with the matrix
  \[
    \Phi \defeq
      \begin{pmatrix}
        0 & 1 & 1 & 1 \\
        1 & 0 & 1 & 1 \\
        1 & 1 & 0 & 1 \\
        1 & 1 & 1 & 0
      \end{pmatrix}
  \]
  of order $2$.
\end{prop}

\begin{proof}
  Let $P \in \Otn_n(\ff{2})$ with $2 \leq n \leq 4$, and $\scal{\cdot}{\cdot}$
  be the symmetric bilinear form defined on $\ff{2}^n$ as
  \[
    \scal{u}{v} \defeq \sum_{i=1}^4 u_i v_i
  \]
  for $u = (u_1, \dots, u_n)$ and $v = (v_1, \dots, v_n)$.
  Every row and every column of $P$ is unitary for
  $\scal{\cdot}{\cdot}$, and thus has an odd number of non-zero coefficients.

  In the case $n=2$, every row and every column of $P$ must have exactly
  one non-zero coefficient, and thus $\Otn_2(\ff{2})$ is reduced to its
  subgroup of permutation matrices.

  If $n=3$, observe that $(1, 1, 1)^\perp$ is null. Every row and every column
  of $P$ must therefore have exactly one non-zero coefficient, and thus $P$ is
  a permutation matrix.
  
  If $n=4$, suppose that there exists $i$ such that the $i$-th row of $P$ has
  exactly one non-zero coefficient.
  Up to multiplying $P$ with permutation matrices, we
  can assume that the first row of $P$ is the vector $(1, 0, 0, 0)$.
  Because this row is orthogonal to the other ones with respect to
  $\scal{\cdot}{\cdot}$, it follows that the first column of $P$ must also
  contain three zeros. In other words we can write
  \[
    P = 
      \begin{pmatrix}
        1 & 0 \\
        0 & Q
      \end{pmatrix}
  \]
  with $Q \in \Otn_3(\ff{2}) \iso \perm_3$, and $P$ is therefore a permutation
  matrix. The same argument applies if there exists $j$ such that the $j$-th
  column of $P$ has only one non-zero coefficient.

  If not, then every row and every column of $P$ has exactly three non-zero
  coefficients. In other words, $P$ is the product of a permutation matrix
  with $\Phi$.

  Therefore $\Otn_4(\ff{2})$ is generated by permutation matrices and $\Phi$.
  A simple computation shows that $\Phi$ commutes with any permutation matrix,
  thus $\Otn_4(\ff{2})$ is isomorphic to the aforementioned product.
\end{proof}

\vspace{\baselineskip}

Now that all the notations are in place, we can explain our strategy.

\subsection*{Overview of the proof}

We apply \cref{thm[1].presentation} to compute $\GW^\sym(R)$
when $(R, \m)$ is a commutative local ring of residue field
$R / \m \iso \ff{2}$.

First, observe that for units $a_1, \dots, b_4$ of $R$, the condition
$(a_1) \perp \cdots \perp (a_4) \iso (b_1) \perp \cdots \perp (b_4)$
can be rephrased as the existence of $P \in \GL_4(R)$ such that
\[
  P \diagMat(a_1, \dots, a_4) \, P^\top = \diagMat(b_1, \dots, b_4)
\]
Because $\ff{2}^\times$ is trivial, this means that $P$ lies over an
orthogonal matrix of $\ff{2}$. But $\Otn_4(\ff{2})$ only contains
permutation matrices and the special matrix $\Phi$ acccording to
\cref{prop[1].otn_cyc2}.

Now let $a_1, \dots, b_4$ be units in $R$ such that
$P \diagMat(a_1, \dots, a_4) \, P^\top = \diagMat(b_1, \dots, b_4)$
for some fixed invertible matrix $P$.
\begin{itemize}
  \setlength\itemsep{.6em}
  \item If $P \equiv \Id_4 \pmod{\m}$, then \cref{thm[2].even} proves that
  the relation
  \[
    \smb{a_1} + \cdots + \smb{a_4} = \smb{b_1} + \cdots + \smb{b_4}
  \]
  in $\Z[R^\times / R^{\times 2}]$ is in the subgroup generated by relations
  from \ref{item[0].even_relation}.

  \item If $P \equiv \Phi \pmod{\m}$, there exists by the
  \cref{lem[3].odd_relation} an invertible matrix $Q \equiv \Phi \pmod{\m}$
  such that
  \[
    Q \diagMat(a_1, \dots, a_4) \, Q^\top = \diagMat(c_1, \dots, c_4)
  \]
  where the relation
  $\smb{a_1} + \cdots + \smb{a_4} = \smb{c_1} + \cdots + \smb{c_4}$ is
  in \ref{item[0].odd_relation}.
  But now $P Q^{-1} \equiv \Id_4 \pmod{\m}$
  and
  \[
    (P Q^{-1}) \diagMat(c_1, \dots, c_4) \, (P Q^{-1})^\top =
      \diagMat(b_1, \dots, b_4)
  \]
  which implies that the relation
  $\smb{c_1} + \cdots + \smb{c_4} = \smb{b_1} + \cdots + \smb{b_4}$
  is in the subgroup generated by \ref{item[0].even_relation} according to
  the first case.
\end{itemize}
Finally the relation
$\smb{a_1} + \cdots + \smb{a_4} = \smb{b_1} + \cdots + \smb{b_4}$
is always contained in the subgroup of $\Z[R^\times / R^{\times 2}]$ generated
by \ref{item[0].even_relation} and \ref{item[0].odd_relation}.

  \section{Congruence over diagonal matrices}

\begin{dfn}
  Let $(R, \m)$ be a commutative local ring, $n \geq 1$ an integer and
  $D = \diagMat(d_1, \dots, d_n)$ a diagonal matrix in $\GL_n(R)$.
  A matrix $P \in \GL_n(R)$ is said to be \nterm{good} relatively to $D$ if
  \begin{enumerate}
    \item $P$ lies over a diagonal matrix in the residue field $R / \m$
    \item off-diagonal coefficients of the matrix, except perhaps two, are zero
    \item the matrix $P D P^\top$ is diagonal.
  \end{enumerate}
\end{dfn}

\begin{discussion}\label{par[2].even}
  Let $(R, \m)$, $n$ and $D$ be as above.
  An explicit computation shows that a matrix $P = (p_{ij})$ is good
  relatively to $D$ if and only if
  \begin{enumerate}
    \item the invertible coefficients of $P$ are exactly the diagonal ones
    \item there exists $1 \leq k < \ell \leq n$ such that
    \begin{enumerate}
      \item all coefficients of $P$ are zero except for the diagonal and maybe
      $p_{k\ell}$ and $p_{\ell k}$
      \item $p_{\ell k} p_{kk} d_k + p_{k\ell} p_{\ell\ell} d_\ell = 0$
    \end{enumerate}
  \end{enumerate}
  Moreover in this case, $P D P^\top = \diagMat(d'_1, \dots , d'_n)$ with
  \[
    d'_i =
      \begin{cases} 
        p_{ii}^2 d_i & \text{if } i \neq k, \ell \\
        p_{kk}^2 d_k + p_{k\ell}^2 d_\ell & \text{if } i=k \\
        p_{\ell\ell}^2 d_\ell + p_{\ell k}^2 d_k & \text{if } i=\ell
      \end{cases}
  \]
  In particular, the congruence $D \sim PDP^\top$ induces the relation
  \[
    \sum_{i=1}^n \smb{d_i} =
      \sum_{i \neq k, \ell} \smb{d_i}
        + \smb[\Big]{d_k + \Big(\frac{p_{k\ell}}{p_{kk}}\Big)^2 d_\ell}
        + \smb[\Big]{d_\ell + \Big(\frac{p_{\ell k}}{p_{\ell\ell}}\Big)^2 d_k}
  \]
  in the symmetric Grothendieck--Witt ring $\GW^\sym(R)$, or equivalently
  \[
    \smb{d_k} + \smb{d_\ell} =
      \smb[\Big]{d_k + \Big(\frac{p_{k\ell}}{p_{kk}}\Big)^2 d_\ell} +
      \smb[\Big]{d_\ell + \Big(\frac{p_{\ell k}}{p_{\ell\ell}}\Big)^2 d_k}
  \]
  Let $I$ be the subgroup of $\Z[R^\times / R^{\times 2}]$ generated by
  the relations 
  \[
    \smb{a} + \smb{b} = \smb{a + n^2 b} + \smb{b + m^2 a}
  \]
  for $a$, $b$ units of $R$, and $m, n \in \m$ such that
  $m a + n b = 0$. In fact $I$ is evidently an ideal.
  The above shows that $I$ is contained in the kernel of the quotient map
  $\Z[R^\times / R^{\times 2}] \twoheadrightarrow \GW^\sym(R)$.

  If $R / \m \iso \ff{2}$, then $I$
  coincides with the subgroup generated by the relations
  \[
    \smb{a} + \smb{b} = \smb{c} + \smb{d}
  \]
  for $(a) \perp (b) \iso (c) \perp (d)$.
  Indeed, $\Otn_2(\ff{2}) \iso \perm_2$ by \cref{prop[1].otn_cyc2}.
\end{discussion}

\vspace{\baselineskip}

We come to the main result of this section :

\begin{thm}\label{thm[2].even}
  Let $(R, \m)$ be a commutative local ring and $n \geq 1$.
  If $D, P \in \GL_n(R)$ are such that
  \begin{enumerate}
    \item $D$ and $PDP^\top$ are diagonal
    \item $P$ lies over a diagonal matrix in the residue field $R / \m$
  \end{enumerate}
  then it is possible to write $P$ as a product
  \[
    P = P_k \cdots P_1
  \]
  such that for $0 \leq i < k$, the matrix $P_{i+1}$ is good with respect to
  $P_i \cdots P_1 D P_1^\top \cdots P_i^\top$.
\end{thm}

\begin{proof}
  The proof is by induction on $n$ :
  \begin{itemize}
    \item the case $n=1$ is clear
    \item suppose that the result holds for $n \geq 1$, and let $D, P \in
    \GL_{n+1}(R)$ as in the statement. Write
    \[
      D = \diagMat(d_1, \dots, d_{n+1})
      \quad\text{and}\quad
      P = 
        \begin{pmatrix}
          P' & U \\
          V^\top & x
        \end{pmatrix}
    \]
    with $P' \in \Mat_{n}(R)$, $U, V \in R^n$ and $x \in R$.
    Observe that $P' \in \GL_n(R)$ and lies over a diagonal matrix in $R / \m$,
    that $x$ is invertible, and that $U, V \equiv 0 \pmod{\m}$.
    From the fact that $PDP^\top$ is diagonal we get the relation:
    \[
      P'D'V + x d_{n+1} U = 0
    \]
    where $D' \defeq \diagMat(d_1, \dots, d_n)$.

    If $V = (v_1, \dots, v_n)$, let $0 \leq k \leq n$ be the smallest integer
    such that $v_i = 0$ for $i > k$. We proceed by induction on $k$ :
    \begin{itemize}
      \item if $k = 0$, then $V = 0$. Therefore $U = 0$ and it suffices
      to apply the outer induction hypothesis to the matrices $D', P' \in \GL_n(R)$
      \item if $k \geq 1$, observe that the matrix
      \[
        E \defeq
          \begin{pmatrix}
            \Id_n & s e_k \\
            v_k e_k^\top & x
          \end{pmatrix}
      \]
      is good with respect to $D$, where $s \defeq -(v_k d_k)/(x d_{n+1})$ and
      $(e_1, \dots, e_n)$ denotes the canonical basis of $R^n$.
      If
      \[
        Q \defeq
          \begin{pmatrix}
            Q' & \frac{1}{x}(U - s Q' e_k) \\
            (V - v_k e_k)^\top & 1
          \end{pmatrix}
      \]
      with $Q' \defeq (xP' - v_k U e_k^\top) (x\Id_n - s v_k e_k e_k^\top)^{-1}$, then
      $P = QE$ and $Q \equiv P \pmod{\m}$. But now
      \[V - v_k e_k = (v_1, \dots, v_{k-1}, 0, \dots, 0)\]
      and we can thus apply the inner induction hypothesis to the couple
      $(Q, EDE^\top)$. \qedhere
    \end{itemize}
  \end{itemize}
\end{proof}

  \section{The presentation}

\begin{lem}\label{lem[3].odd_relation}
  Let $(R, \m)$ be a commutative local ring with residue field
  $R / \m \iso \ff{2}$.
  For any diagonal matrix $D = \diagMat(a, b, c, d)$ in $\GL_4(R)$,
  there exists $P \in \GL_4(R)$ such that
  \begin{enumerate}
    \item $P \equiv \Phi \pmod{\m}$
    \item $PDP^\top = \diagMat(t, u, v, w)$
  \end{enumerate}
  with
  \begin{align*}
    \smb{t} &= \smb{a b c d \cdot u v w} \\
    \smb{u} &= \smb[\Big]{\,a + \frac{1}{c} + \frac{1}{d}\,} \\
    \smb{v} &= \smb[\Big]{\,\frac{1}{a} + \frac{1}{b} + d\,} \\
    \smb{w} &= \smb[\big]{\,a + b + a^2 c\,}
  \end{align*}
  in $R^\times / R^{\times 2}$. 
\end{lem}

\begin{proof}
  For $p$, $q$ and $r$ three units in $R$, let
  \[
    P^\top \defeq 
      \begin{pmatrix}
        \frac{1}{r^2}\frac{a c}{b d} - r^2 &
          -\frac{r d}{p a} &
          p &
          -\frac{q r}{p} \frac{b}{a d} \\

        -\frac{p}{q r^2}\frac{a^2 c}{b^2 d} - \frac{1}{p q} \frac{a c}{b^2} - \frac{q r^2}{p} &
        0 &
        q &
        r\frac{1}{d} \\

        -\frac{p^2}{q r} \frac{a^2}{b d} - \frac{qa}{rd} - \frac{ra}{qb} &
        \frac{q r^2}{p} \frac{b d}{a c} &
        0 &
        -\frac{1}{p d} \\

        p r \frac{a}{d} + \frac{1}{p r} \frac{a c}{b d} + \frac{q^2 r}{p} \frac{b}{d} &
        1 &
        r &
        0
      \end{pmatrix}
  \]
  Evidently $P \equiv \Phi \pmod{\m}$, in particular $P$ is invertible.
  We leave the reader to check, hopefully with a computer, that
  $PDP^\top = \diagMat(t, u, v, w)$ with
  \begin{align*}
    t &= \frac{p^2}{q^2 r^4}\frac{a^3 c}{b^3 d} u v w \\
    u &= \frac{r^2 d^2}{p^2 a^2} \Big( a + \frac{p^2}{r^2} \frac{a^2}{d} + q^2 r^2 \frac{b^2}{c} \Big) \\
    v &= p^2 a + q^2 b + r^2 d \\
    w &= \frac{r^2}{p^2 a^2 d^2} \Big( p^2 a^2 b + \frac{1}{r^2} a^2 c + q^2 a b^2 \Big)
  \end{align*}
  Therefore:
  \begin{align*}
    \smb{t} &= \smb{a b c d \cdot u v w}\\
    \smb{u} &= \smb[\Big]{\,a + \frac{p^2}{r^2} \frac{a^2}{d} + q^2 r^2 \frac{b^2}{c}\,} \\
    \smb{v} &= \smb[\big]{\,p^2 a + q^2 b + r^2 d\,} \\
    \smb{w} &= \smb[\Big]{\,p^2 a^2 b + \frac{1}{r^2} a^2 c + q^2 a b^2\,}
  \end{align*}
  The choice $p = 1/a$, $q = 1/b$, $r = 1$ thus concludes the proof.
  Observe that the equality $\smb{t} = \smb{a b c d \cdot u v w}$ was in fact
  automatic since $\smb{\,\det{D}\,} = \smb{\,\det{PDP^\top}\,}$.
\end{proof}

\vspace{\baselineskip}

We can finally prove our main result :

\begin{thm}\label{thm[3].presentation}
  Let $R$ be a commutative local ring with residue field $\ff{2}$.
  Then the kernel of the surjective ring map
  \[
    \Z[R^\times / R^{\times 2}] \twoheadrightarrow \GW^\sym(R)
  \]
  is additively generated by the following relations :
  \begin{enumerate}[label=(\arabic*)]
    \item\label{item[3].even_relation} for units $a$, $b$
    and $m$, $n$ non-invertible such that $ma + nb = 0$ :
    \[
      \smb{a} + \smb{b} = \smb{a + n^2 b} + \smb{b + m^2 a}
    \]

    \item\label{item[3].odd_relation} for units $a$, $b$, $c$ and $d$ :
    \[
      \smb{a} + \smb{b} + \smb{c} + \smb{d} =
        \smb{u} + \smb{v} + \smb{w} + \smb{a b c d \cdot u v w}
    \]
    where $u \defeq a + \frac{1}{c} + \frac{1}{d}$, $v \defeq \frac{1}{a} +
    \frac{1}{b} + d$ and $w \defeq a + b + a^2 c$.
  \end{enumerate}
\end{thm}

\begin{proof}
  The discussion at \ref{par[2].even} and \cref{lem[3].odd_relation}
  respectively show that relations \ref{item[3].even_relation} and
  \ref{item[3].odd_relation} are statisfied in the symmetric
  Grothendieck--Witt ring $\GW^\sym(R)$.
  The kernel of the quotient map
  $\Z[R^\times / R^{\times 2}] \twoheadrightarrow \GW^\sym(R)$
  thus contains the subgroup $I$ additively generated by those relations.

  We now turn to the converse inclusion. 
  According to \cref{thm[1].presentation},
  it suffices to prove that $I$ contains relations
  \[
    \smb{a_1} + \cdots + \smb{a_4} = \smb{b_1} + \cdots + \smb{b_4}
  \]
  for units $a_1, \dots, b_4$ of $R$ and $P \in \GL_4(R)$
  such that $P \diagMat(a_1, \dots, a_4)\, P^\top = \diagMat(b_1, \dots, b_4)$.

  Choose such $a_1, \dots, b_4$ and $P$, and set
  \[
    D \defeq \diagMat(a_1, \dots, a_4)
    \quad\text{and}\quad
    D' \defeq \diagMat(b_1, \dots, b_4)
  \]
  so that $P D P^\top = D'$.
  As was remarked in \ref{rmk[1].base_change_otn}, the matrix $P$ must lie over
  an orthogonal matrix of $\ff{2}$, and according to
  \cref{prop[1].otn_cyc2} there exists a permutation
  $\sigma \in \perm_4$ such that
  \[
    P \sigma \equiv \Id_4 \text{ or } \Phi \pmod{\m}
  \]
  where $\m$ denotes the maximal ideal of $R$ and identifying
  permutations with permutation matrices.
  Observe that
  $\sigma D \sigma^\top = \diagMat(a_{\sigma(1)}, \dots, a_{\sigma(4)})$,
  but
  \[
    \smb{a_1} + \cdots + \smb{a_4} =
      \smb{a_{\sigma(1)}} + \cdots + \smb{a_{\sigma(4)}}
  \]
  in $\Z[R^\times / R^{\times 2}]$, and thus also in the quotient group
  $\Z[R^\times / R^{\times 2}] / I$,
  because addition is commutative.
  We can therefore assume that $\sigma = 1$. Then :
  \begin{itemize}
    \setlength\itemsep{.6em}
    \item if $P \equiv \Id_4 \pmod{\m}$, \cref{thm[2].even} applies and
    the relation
    \[
      \smb{a_1} + \cdots + \smb{a_4} = \smb{b_1} + \cdots + \smb{b_4}
    \]
    is contained in the subgroup generated by
    \ref{item[3].even_relation} according to \ref{par[2].even}.

    \item if $P \equiv \Phi \pmod{\m}$, there exists
    by \cref{lem[3].odd_relation} an invertible $Q \equiv \Phi \pmod{\m}$
    such that
    \[
      QDQ^\top = \diagMat(c_1, \dots, c_4)
    \]
    and such that the relation
    $\smb{a_1} + \cdots + \smb{a_4} = \smb{c_1} + \cdots + \smb{c_4}$
    is in \ref{item[3].odd_relation}.
    But now observe that $P Q^{-1} \equiv \Id_4 \pmod{\m}$ and that
    \[
      (PQ^{-1}) \diagMat(c_1, \dots, c_4)\, (PQ^{-1})^\top = D'
    \]
    By the first case, this implies that the relation
    $\smb{c_1} + \cdots + \smb{c_4} = \smb{b_1} + \cdots + \smb{b_4}$
    is in the subgroup generated by \ref{item[3].even_relation}.
    Finally the relation
    $\smb{a_1} + \cdots + \smb{a_4} = \smb{b_1} + \cdots + \smb{b_4}$
    is in $I$. \qedhere
  \end{itemize}
\end{proof}

\vspace{.1em}

\begin{rmk}
  Relation \ref{item[3].even_relation} alone does not suffice in general, see
  for instance \cref{prop[4].calcul_cyc4_cyc8} for the computation
  of $\GW^\sym(\cyc{4}) \iso \Z \oplus \cyc{4}$.
  In particular $\GW^\sym(\cyc{4})$ is not freely generated by $\smb{1}$
  and $\smb{3}$, in spite of the fact that any relation from
  \ref{item[3].even_relation} is trivial in this case.
\end{rmk}

\begin{rmk}\label{rmk[3].rel_hyp}
  Relation \ref{item[3].odd_relation} applied to the tuple $(1, a, -1, -1)$
  for some unit $a$ recovers the familiar relation
  \[
    \smb{1} + \smb{-1} = \smb{a} + \smb{-a}
  \]
  in $\GW^\sym(R)$ of a commutative local ring $R$.
\end{rmk}

\begin{rmk}
  The symmetric Witt ring $\W^\sym(R)$ of a commutative local ring $R$
  is obtained by quotienting $\GW^\sym(R)$ by the subgroup generated
  by the hyperbolic element $\smb{1} + \smb{-1}$;
  see \cite{KRW72}*{corollary 1.17} for instance.
  We deduce from this fact and from \cref{thm[3].presentation} an explicit
  presentation of $\W^\sym(R)$ as an abelian group when $R$ has residue
  field $\ff{2}$. Namely, the kernel of the surjective ring map
  \[
    \Z[R^\times / R^{\times 2}] \twoheadrightarrow \W^\sym(R)
  \]
  is additively generated by relations \ref{item[3].even_relation},
  \ref{item[3].odd_relation} and $\smb{1} + \smb{-1} = 0$.
\end{rmk}

  \section{Applications}

A simple application of \cref{thm[3].presentation} is the following
criterion, which we will leverage to compute the symmetric Grothendieck--Witt
groups of the rings $\cyc{2^n}$ and $\ff{2}[x] / (x^n)$.

Given a unit $a$ inside a commutative ring $R$, we note
$\pfister{a} \defeq \smb{1} - \smb{a} \in \GW^\sym(R)$.

\begin{prop}\label{prop[4].criterion}
  A surjective morphism $\phi\colon R \rightarrow S$
  between commutative local rings of residue field $\F$
  induces a surjective morphism
  \[
    \phi_*\colon \GW^\sym(R) \twoheadrightarrow \GW^\sym(S)
  \]
  whose kernel is additively generated by elements
  $\smb{a} \pfister{x}$ with units $a$ and $x$
  such that $\phi(x) = 1$.
  It is in particular an isomorphism when $\phi$ restricts
  to an isomorphism
  $R^\times / R^{\times 2} \iso S^\times / S^{\times 2}$.
\end{prop}

\begin{proof}
  Observe that $\phi$ is necessarily local because $R$ and $S$ share the same
  residue field. In particular the induced map
  $R^\times / R^{\times 2} \rightarrow S^\times / S^{\times 2}$ is surjective,
  and it follows from \cref{thm[1].presentation}
  that
  \[
    \phi_*\colon \GW^\sym(R) \rightarrow \GW^\sym(S)
  \]
  is surjective as well.
  From the surjectivity of $R^{\times 2} \rightarrow S^{\times 2}$,
  one sees that it suffices to prove that $\kera{\phi_*}$ is generated by those
  $\smb{a} \pfister{x}$ where $x$ is such that $\phi(x)$ is a square,
  or equivalently by $\smb{a} - \smb{b}$ with units $a$ and $b$ such that
  $\smb{\phi(a)} = \smb{\phi(b)}$.

  The snake lemma applied to the following commutative diagram
  \begin{center}
    \begin{tikzcd}[sep=.7cm]
      0
        \arrow[r] &
      I
        \arrow[r]
        \arrow[d, dashrightarrow, "q"] &
      \Z[R^\times / R^{\times 2}]
        \arrow[r]
        \arrow[d, twoheadrightarrow, "p_R"] &
      \Z[S^\times / S^{\times 2}]
        \arrow[d, twoheadrightarrow, "p_S"]
        \arrow[r] &
      0 \\
      0
        \arrow[r] &
      \kera{\phi_*}
        \arrow[r] &
      \GW^\sym(R)
        \arrow[r, "\phi_*"] &
      \GW^\sym(S)
        \arrow[r] &
      0
    \end{tikzcd}
  \end{center}
  yields an exact sequence
  \begin{center}
    \begin{tikzcd}[sep=.7cm]
      \kera{p_R}
        \arrow[r, "\smb{\phi}"] &
      \kera{p_S}
        \arrow[r] &
      \cokera{q}
        \arrow[r] &
      \cokera{p_R} \iso 0
    \end{tikzcd}
  \end{center}
  The statement thus reduces to the surjectivity of
  $\smb{\phi}\colon \kera{p_R} \rightarrow \kera{p_S}$.

  We treat the case $\F \iso \ff{2}$, the other case is
  similar using instead the presentation found in \cite{RS24}.
  By \cref{thm[3].presentation},
  $\kera{p_S}$ is generated as a subgroup of $\Z[S^\times / S^{\times 2}]$ by
  the following families of relations :
  \begin{enumerate}[label=(\arabic*)]
    \item\label{item[4].even_relation} for units $a$, $b$
    and $m$, $n$ non-invertible in $S$ such that $ma + nb = 0$ :
    \[
      \smb{a} + \smb{b} = \smb{a + n^2 b} + \smb{b + m^2 a}
    \]

    \item\label{item[4].odd_relation} for units $a$, $b$, $c$ and $d$ :
    \[
      \smb{a} + \smb{b} + \smb{c} + \smb{d} =
        \smb{u} + \smb{v} + \smb{w} + \smb{a b c d \cdot u v w}
    \]
    where $u$, $v$ and $w$ are explicit polynomials in $a$, $b$, $c$ and $d$.
  \end{enumerate}
  We want to show that these relations lift to $\kera{p_R}$. Let
  \[
    \smb{a} + \smb{b} = \smb{a + n^2 b} + \smb{b + m^2 a}
  \]
  be a relation from \ref{item[4].even_relation}.
  Since $\phi$ is both local and surjective, there exists units $a'$ and $b'$
  in $R$ as well as $m'$ non-invertible lifting the elements $a$, $b$ and $m$.
  But then $n' \defeq - m' a' / b'$ lifts $n$ and is non-invertible,
  and \cref{thm[3].presentation} implies that the relation
  \[
    \smb{a'} + \smb{b'} = \smb{a' + (n')^2 b'} + \smb{b' + (m')^2 a'}
  \]
  is in $\kera{p_R}$. This furnishes the required lift.
  
  The same argument shows
  that relations from \ref{item[4].odd_relation} lift to $\kera{p_R}$ as well.
\end{proof}

\vspace{\baselineskip}

\subsection{Cyclic groups}

\begin{prop}\label{prop[4].calcul_cyc4_cyc8}
  We have group isomorphisms
  \[
      \GW^\sym(\cyc{4}) \iso \Z \oplus \cyc{4}
      \quad\text{and}\quad
      \GW^\sym(\cyc{8}) \iso \Z \oplus \cyc{4} \oplus \cyc{2}
  \]
  of inverses respectively induced by the elements
  $(\smb{1}, \pfister{3})$ and
  $(\smb{1}, \pfister{3}, \pfister{5})$.

  In particular $\W^\sym(\cyc{4}) \iso \cyc{8}$
  and $\W^\sym(\cyc{8}) \iso \cyc{8} \oplus \cyc{2}$
  via respectively $\smb{1}$ and $(\smb{1}, \pfister{5})$.
\end{prop}

\begin{proof}
  According to \cref{thm[3].presentation}, the ring $\GW^\sym(\cyc{4})$ is
  generated as a group by the symbols $\smb{1}$ and $\smb{3}$
  with the relation $4\smb{1} = 4\smb{3}$.

  Similarly, one can check that $\GW^\sym(\cyc{8})$
  is generated by the symbols $\smb{1}$, $\smb{3}$, $\smb{5}$ and $\smb{7}$
  with the following relations:
  \begin{multicols}{3}
    \begin{enumerate}[label=\alph*)]
      \setlength\itemsep{.3em}
      \item\label{item[4].cyc8_rel_a}
        $\smb{1} + \smb{3} = \smb{5} + \smb{7}$
      \item\label{item[4].cyc8_rel_b}
        $\smb{1} + \smb{7} = \smb{3} + \smb{5}$
      \item\label{item[4].cyc8_rel_c}
        $2\smb{1} = 2\smb{5}$
      \item\label{item[4].cyc8_rel_d}
        $2\smb{3} = 2\smb{7}$
      \item\label{item[4].cyc8_rel_e}
        $4\smb{1} = 4\smb{3}$
      \item\label{item[4].cyc8_rel_f}
        $4\smb{5} = 4\smb{7}$
      \item\label{item[4].cyc8_rel_g}
        $2(\smb{1} + \smb{5}) = 2(\smb{3} + \smb{7})$
      \item\label{item[4].cyc8_rel_h}
        $3\smb{1} + \smb{5} = \smb{3} + 3\smb{7}$
      \item\label{item[4].cyc8_rel_i}
        $\smb{1} + 3\smb{3} = \smb{5} + 3\smb{7}$
      \item\label{item[4].cyc8_rel_j}
        $\smb{1} + 3\smb{5} = 3\smb{3} + \smb{7}$
      \item\label{item[4].cyc8_rel_k}
        $3\smb{1} + \smb{3} = 3\smb{5} + \smb{7}$
    \end{enumerate}
  \end{multicols}
  \noindent By substituting the relation $\smb{7} = \smb{3} + \smb{5} - \smb{1}$
  deduced from \ref{item[4].cyc8_rel_b}, the relations
  \ref{item[4].cyc8_rel_a}, \ref{item[4].cyc8_rel_d},
  \ref{item[4].cyc8_rel_f}, \ref{item[4].cyc8_rel_g},
  \ref{item[4].cyc8_rel_h}, \ref{item[4].cyc8_rel_i},
  \ref{item[4].cyc8_rel_j} and \ref{item[4].cyc8_rel_k}
  are all seen to be implied by
  \ref{item[4].cyc8_rel_c} and \ref{item[4].cyc8_rel_e}.
  In particular $\GW^\sym(\cyc{8})$ is generated by $\smb{1}$, $\smb{3}$,
  $\smb{5}$ and $\smb{7}$ and relations \ref{item[4].cyc8_rel_b},
  \ref{item[4].cyc8_rel_c} and \ref{item[4].cyc8_rel_e}, or equivalently
  by $\smb{1}$, $\smb{3}$, $\smb{5}$ and relations \ref{item[4].cyc8_rel_c}
  and \ref{item[4].cyc8_rel_e}.

  The statement for symmetric Witt groups is then obtained by enforcing the
  relation $\smb{1} + \smb{-1} = 0$.
  In $\GW^\sym(\cyc{8})$, it is expressed as $\smb{3} + \smb{5} = 0$ thanks
  to \ref{item[4].cyc8_rel_b}.
\end{proof}

\begin{rmk}\label{rmk[4].anneau_gw_cyc4_cyc8}
  It is possible to make the ring structure on $\GW^\sym(\cyc{4})$
  and $\GW^\sym(\cyc{8})$ explicit using \cref{prop[4].calcul_cyc4_cyc8},
  since it suffices to compute the products of the exhibited generators.
  Explicitly, we obtain ring isomorphisms
  \[
    \GW^\sym(\cyc{4}) \iso \Z[x] / (4x, x^2 - 2x)
    \quad\text{and}\quad
    \GW^\sym(\cyc{8}) \iso \Z[x, y] / (4x, 2y, x^2 - 2x, y^2, xy)
  \]
  and
  \[
    \W^\sym(\cyc{4}) \iso \cyc{8}
    \quad\text{and}\quad
    \W^\sym(\cyc{8}) \iso (\cyc{8})[y] / (2y, y^2)
  \]
  where $x$ and $y$ are sent to $\pfister{3}$ and $\pfister{5}$ respectively.
  The projections $\GW^\sym(\cyc{4}) \twoheadrightarrow \W^\sym(\cyc{4})$
  and $\GW^\sym(\cyc{8}) \twoheadrightarrow \W^\sym(\cyc{8})$
  then correspond under these identifications to the evaluations at $x = 2$
  and $x = 2 - y$.
\end{rmk}

\begin{cor}\label{cor[4].tower_cyc2n}
  The canonical projections induce a tower of isomorphisms
  \begin{center}
    \begin{tikzcd}[sep=.6cm]
      \GW^\sym(\cyc{8}) &
      \GW^\sym(\cyc{16})
        \arrow[l, swap, "\iso"] &
      \cdots
        \arrow[l, swap, "\iso"] &
      \GW^\sym(\cyc{2^n})
        \arrow[l, swap, "\iso"] &
      \cdots
        \arrow[l, swap, "\iso"] &
      \GW^\sym(\Z_2)
        \arrow[l, swap, "\iso"]
    \end{tikzcd}
  \end{center}
  where $\Z_2$ is the ring of $2$-adic integers.
  In particular we have group isomorphisms
  \[
    \GW^\sym(\cyc{2^n}) \iso \GW^\sym(\Z_2) \iso
      \Z \oplus \cyc{4} \oplus \cyc{2}
    \quad\text{and}\quad
    \W^\sym(\cyc{2^n}) \iso \W^\sym(\Z_2) \iso
      \cyc{8} \oplus \cyc{2}
  \]
  for $n \geq 3$.
\end{cor}

\begin{proof}
  For $n \geq 2$, recall that we have a group isomorphism
  \cite{IR90}*{theorem 2'}
  \[
    \begin{pmatrix}
      -1 \\ 5
    \end{pmatrix}\colon \cyc{2} \oplus \cyc{2^{n-2}} \iso
      (\cyc{2^n})^\times
  \]
  and thus
  $(\cyc{2^n})^\times / (\cyc{2^n})^{\times 2} \iso \cyc{2} \oplus \cyc{2}$ for all $n \geq 3$.

  Moreover, the functor $(-)^\times$ that assigns to a commutative ring its
  abelian group of units is right adjoint to the group algebra functor $\Z[-]$,
  and thus preserves all limits.
  In particular
  \begin{align*}
    \Z_2^\times &\iso \lim_{n \geq 2} (\cyc{2^n})^\times \\
      &\iso \cyc{2} \oplus \lim_n \cyc{2^n} \\
      &\iso \cyc{2} \oplus \Z_2
  \end{align*}
  Therefore $\Z_2^\times / \Z_2^{\times 2} \iso \cyc{2} \oplus \cyc{2}$.

  Because all projection maps are surjective and local, the maps induced
  on square classes are surjective as well.
  They therefore must be isomorphisms by the above,
  and \cref{prop[4].criterion} applies.
  The second assertion then follows from \cref{prop[4].calcul_cyc4_cyc8}.
\end{proof}

\vspace{.1em}

The above computations allow us to describe the kernel and
the cokernel of the symmetrisation map for the rings $\cyc{2^n}$
and $\Z_2$.
In the following, $\W^\qud(-)$ will denote the quadratic
Witt ring functor.

\begin{lem}\label{lem[4].complete_symmetrisation}
  If $R$ is a complete local ring of residue field $\ff{2}$, then
  the symmetrisation map
  \[
    \cyc{2} \iso \W^\qud(R) \rightarrow \W^\sym(R)
  \]
  is induced by the element $3\smb{1} - \smb{3}$.
\end{lem}

\begin{proof}
  By \cite{Wal73}*{lemma 2.5}, the projection maps induce
  isomorphisms
  \[
    \W^\qud(\Z_2) \iso \W^\qud(\cyc{2^n}) \iso \W^\qud(\ff{2})
  \]
  and the Arf invariant gives an isomorphism $\cyc{2} \iso \W^\qud(\ff{2})$
  \cite{MH73}*{appendix 1}. More explicitly, it is induced by the element
  \[
    \alpha \defeq
      \smb[\Big]{
        \begin{pmatrix}
          1 & 1 \\
          0 & 1
        \end{pmatrix}
      }
  \]
  Here $\alpha$ is the unimodular quadratic form on $\ff{2}^2$ given
  by the formula $(x, y) \mapsto x^2 + y^2 + xy$.
  Lift $\alpha$ to the element $\beta \in \W^\qud(R)$ defined as
  \[
    \beta \defeq
      \smb[\Big]{
        \begin{pmatrix}
          1 & 1 \\
          0 & 1
        \end{pmatrix}
      }
  \]
  so that the identification $\cyc{2} \iso \W^\qud(R)$ is induced by $\beta$.
  The symmetrisation map $\sigma\colon \W^\qud(R) \rightarrow \W^\sym(R)$ then
  sends $\beta$ to
  \begin{align*}
    \sigma(\beta) &=
      \smb[\Big]{
        \begin{pmatrix}
          2 & 1 \\
          1 & 2
        \end{pmatrix}
      } \\
      &= \smb{-3} + 2\smb{1} - \smb{-1} \\
      &= 3\smb{1} - \smb{3}
  \end{align*}
  Here, the first step relies on \cite{RS24}*{proposition 3.1, (2)}
\end{proof}

\begin{cor}
  For $n \geq 2$, the symmetrisation map
  \[
    \W^\qud(\cyc{2^n}) \rightarrow \W^\sym(\cyc{2^n})
  \]
  is injective, and has cokernel
  \[
    \W^\sym(\cyc{2^n}) / \W^\qud(\cyc{2^n}) \iso
      \begin{cases} 
        \cyc{4} & \text{if } n = 2  \\
        \cyc{8} & \text{if } n \geq 3
      \end{cases}
  \]
  Similarly, the symmetrisation map
  $\W^\qud(\Z_2) \rightarrow \W^\sym(\Z_2)$
  also has cokernel $\cyc{8}$.
\end{cor}

\begin{proof}
  It suffices to prove the result for $\cyc{4}$ and $\cyc{8}$
  according to \cref{cor[4].tower_cyc2n}.
  
  Coupling \cref{lem[4].complete_symmetrisation} with
  \cref{prop[4].calcul_cyc4_cyc8}, the image of the symmetrisation
  map
  \[
    \cyc{2} \iso \W^\qud(\cyc{8}) \rightarrow \W^\sym(\cyc{8})
  \]
  is induced by the element
  \begin{align*}
    3\smb{1} - \smb{3} &= 3\smb{1} + \smb{5} \\
      &= 4\smb{1} - \pfister{5}
  \end{align*}
  which corresponds to $(4, 1)$ under the identification
  $\W^\sym(\cyc{8}) \iso \cyc{8} \oplus \cyc{2}$.
  In particular the symmetrisation map is not trivial, thus
  injective, and a change of coordinates gives the announced cokernel.

  Mapping those computations along the projection $\cyc{8} \rightarrow \cyc{4}$
  shows that the symmetrisation map is induced by $4\smb{1}$ in
  $\W^\sym(\cyc{4}) \iso \cyc{8}$.
\end{proof}

\vspace{\baselineskip}

\subsection{Truncated polynomials}

In this section, we compute the symmetric Grothendieck--Witt rings
of $\ff{2}[x] / (x^n)$ for $n \geq 2$ using two methods.
One uses the presentation established in \cref{thm[3].presentation},
and serves as an illustration of \cref{prop[4].criterion}.
The other is more direct, and relies on the following :

\begin{prop}\label{prop[4].fund_power_series}
  The fundamental ideals of $\W^\sym(\ff{2}\llbracket x \rrbracket)$
  and $\W^\sym(\ff{2}(\!( x )\!))$ both square to zero.
\end{prop}

\begin{proof}
  The ring of formal power series
  $\ff{2}\llbracket x \rrbracket \defiso \lim_n \ff{2}[x] / (x^n)$
  is a discrete valuation ring, and thus the canonical morphism
  $\W^\sym(\ff{2}\llbracket x \rrbracket)
    \rightarrow \W^\sym(\ff{2}(\!( x )\!))$
  is injective \cite{MH73}*{corollary IV.3.3}.
  It thus suffices to prove the statement for $\ff{2}(\!( x )\!)$.
  But this follows from \cite{MH73}*{theorem III.5.10},
  since the Frobenius $\ff{2}(\!( x )\!) \hookrightarrow \ff{2}(\!( x )\!)$
  is a quadratic field extension.
  Indeed, squares in $\ff{2}(\!( x )\!)$ are the even Laurent series.
\end{proof}

\begin{cor}\label{cor[4].pfister_power_series}
  All $2$-fold Pfister forms vanish in
  $\GW^\sym(\ff{2}\llbracket x \rrbracket)$ and $\GW^\sym(\ff{2}(\!( x )\!))$.
\end{cor}

\begin{proof}
  For $u$ and $v$ units in $\ff{2}\llbracket x \rrbracket$
  or $\ff{2}(\!( x )\!)$, the form
  \[
    \pfister{u, v} \defeq (\smb{1} - \smb{u})(1 - \smb{v})
  \]
  has trivial rank and Witt class by \cref{prop[4].fund_power_series}.
\end{proof}

\begin{rmk}\label{rmk[4].pfister_sqz}
  If $R$ is a commutative local ring of characteristic $2$
  for which $2$-fold Pfister forms vanish in $\GW^\sym(R)$,
  then using the relations
  \[
    \smb{a} + \smb{b} = \smb{1} + \smb{ab}
    \quad\text{and}\quad
    - \smb{a} = - 2\smb{1} + \smb{a}
  \]
  one can show that every element of $\GW^\sym(R)$ can be
  decomposed, necessarily uniquely, as $k\smb{1} + \pfister{a}$
  for some integer $k$ and some unit $a$.
  In other words, the ring $\GW^\sym(R)$ is a square zero extension
  \[
    \GW^\sym(R) \iso \Z \oplus R^\times / R^{\times 2}
  \]
  \Cref{cor[4].pfister_power_series} implies that
  this holds true for $\ff{2}(\!( x )\!)$, $\ff{2}\llbracket x \rrbracket$,
  as well as $\ff{2}[x]/(x^n)$ for $n \geq 1$.
\end{rmk}

\begin{lem}\label{lem[4].sq_trunc_exact_seq}
  For $n \geq 1$, let $R_n \defiso \ff{2}[x]/(x^n)$.
  The canonical projection $R_{n+1} \rightarrow R_n$
  sits inside an exact sequence
  \begin{center}
    \begin{tikzcd}[sep=.7cm]
      \cyc{2}
        \arrow[r] &
      R_{n+1}^\times / R_{n+1}^{\times 2}
        \arrow[r] &
      R_n^\times / R_n^{\times 2}
        \arrow[r] &
      0
    \end{tikzcd}
  \end{center}
  where the left map picks out the element $\smb{1+x^n}$.
\end{lem}

\begin{proof}
  Let $\pi\colon R_{n+1} \rightarrow R_n$ denote the projection,
  and $\smb{u} \in R_{n+1}^\times / R_{n+1}^{\times 2}$ such that
  $\smb{\pi(u)} = \smb{1}$.
  Since $\pi$ is surjective on units, it is surjective on square of units
  as well, and therefore one can assume that $\pi(u) = 1$.
  Therefore $\smb{u}$ must be $\smb{1}$ or $\smb{1+x^n}$.
\end{proof}

\begin{prop}\label{prop[4].basis_sq_trunc}
  Let $n \geq 1$ and $k$ the greatest integer such that $2k+2 \leq n$.
  The family
  \[
    \smb{1+x}, \smb{1+x^3}, \dots, \smb{1+x^{2k+1}}
  \]
  forms a basis of the $\ff{2}$-module of square classes
  of $\ff{2}[x]/(x^n)$.
\end{prop}

\begin{proof}
  For $n \geq 1$, let $R_n \defiso \ff{2}[x]/(x^n)$.
  The left map in the exact sequence provided by
  \cref{lem[4].sq_trunc_exact_seq}
  \begin{center}
    \begin{tikzcd}[sep=.7cm]
      \cyc{2}
        \arrow[r] &
      R_{n+1}^\times / R_{n+1}^{\times 2}
        \arrow[r] &
      R_n^\times / R_n^{\times 2}
        \arrow[r] &
      0
    \end{tikzcd}
  \end{center}
  is zero exactly when $1 + x^n$ is a square, or in other words exactly
  when $n$ is even.

  From this remark, the proof simply proceeds by induction on $n$:
  \begin{itemize}
    \item for $n = 1$, the statement follows from
    $\ff{2}^\times / \,\ff{2}^{\times 2} \iso 0$

    \item for $n$ even, the morphism
    \[
      R_{n+1}^\times / R_{n+1}^{\times 2}
        \rightarrow R_n^\times / R_n^{\times 2}
    \]
    is an isomorphism

    \item if the statement holds for $R_n^\times / R_n^{\times 2}$
    for some odd integer $n = 2k+1$, then the exact sequence
    \begin{center}
      \begin{tikzcd}[sep=.7cm]
        0
          \arrow[r] &
        \cyc{2}
          \arrow[r] &
        R_{n+1}^\times / R_{n+1}^{\times 2}
          \arrow[r] &
        R_n^\times / R_n^{\times 2}
          \arrow[r] &
        0
      \end{tikzcd}
    \end{center}
    is seen to be split via the explicit morphism
    \[
      R_n^\times / R_n^{\times 2} \rightarrow
        R_{n+1}^\times / R_{n+1}^{\times 2}
    \]
    sending each basis element $\smb{1 + x^{2\ell + 1}}$
    to $\smb{1 + x^{2\ell + 1}}$.
    \qedhere
  \end{itemize}
\end{proof}

\begin{rmk}
  The proof of \cref{prop[4].basis_sq_trunc} shows that one can in fact
  replace the generators $\smb{1 + x^{2\bullet + 1}}$ in the statement by
  $\smb{1 + x^{2\bullet + 1} u_\bullet}$ for any choice of sequence of
  units $u_\bullet$ in $\ff{2}\llbracket x \rrbracket$.
\end{rmk}

\begin{cor}\label{cor[4].calcul_trunc}
  The elements $\smb{1}, \pfister{1+x}, \pfister{1+x^3}, \dots$
  induce ring isomorphisms
  \[
    \GW^\sym\!\big(\ff{2}[x] / (x^{2n}) \big)
      \iso \GW^\sym\!\big(\ff{2}[x] / (x^{2n+1}) \big)
      \iso \Z \oplus (\cyc{2})^{n}
  \]
  where the right-hand side is a square zero extension of $\Z$.
\end{cor}

\begin{proof}[Proof 1]
  This follows immediately from \cref{rmk[4].pfister_sqz}
  and \cref{prop[4].basis_sq_trunc}.
\end{proof}

\begin{proof}[Proof 2]
  Let us sketch how one can obtain such a result without knowing
  \cref{prop[4].fund_power_series}. Similarly to the proof of
  \cref{prop[4].basis_sq_trunc}, it will suffice to establish that
  the morphism induced by canonical projection
  $\pi \colon \ff{2}[x] / (x^{n+1}) \rightarrow \ff{2}[x] / (x^n)$
  sits inside an exact sequence
  \begin{center}
    \begin{tikzcd}[sep=.7cm]
      \cyc{2}
        \arrow[r] &
      \GW^\sym\!\big(\ff{2}[x] / (x^{n+1}) \big)
        \arrow[r, "\pi_*"] &
      \GW^\sym\!\big(\ff{2}[x] / (x^n) \big)
        \arrow[r] &
      0
    \end{tikzcd}
  \end{center}
  for $n \geq 1$, where the left map picks out the $2$-torsion
  element $\pfister{1+x^n}$.

  According to \cref{prop[4].criterion}, the kernel of $\pi_*$
  is generated by the elements $\smb{u} \pfister{x}$ such that $\pi(x) = 1$,
  or in other words by the elements $\smb{u} \pfister{1+x^n}$.
  Since \cref{rmk[3].rel_hyp} already implies that $\pfister{1+x^n}$ is
  $2$-torsion, proving that
  \[
    \smb{u} \pfister{1+x^n} = \pfister{1+x^n}
    \quad\text{or equivalently}\quad
    \pfister{u, 1+x^n} = 0
  \]
  for all units $u$ will finish the proof.
  This is of course true by \cref{rmk[4].pfister_sqz}, but there is a more
  direct argument.
  First, observe that this is clear when $n$ is even or
  when $n = 1$.

  Assume now that $n = 2k+1$ is odd with $k \geq 1$.
  The units $u$ satisfying the equation are evidently stable under product,
  and if $u$ is of the form $1 + tx^2$ for some polynomial $t$, note that
  \[
    u = (1 + x )\Big(1 + x + \frac{1+t}{1+x} x^2 \Big)
  \]
  It is therefore sufficient to deal with the case
  where $u$ can be written $1 + x + tx^2$ for some polynomial $t$.
  Under this assumption, substituting
  $(1, u (1+x^{2k+1}), x^k)$ to $(a, b, n)$ in relation
  \ref{item[3].even_relation} from \cref{thm[3].presentation} yields :
  \begin{align*}
    \smb{1} + \smb{u (1+x^n)}
      &= \smb{1 + u x^{2k} (1 + x^{2k+1})}
        \big( \smb{1} + \smb{u (1+x^n)} \big) \\
      &= \smb{1 + x^{2k} + x^{2k+1}} \big( \smb{1} + \smb{u (1+x^n)} \big) \\
      &= \smb{(1+x^k)^2 (1 + x^{2k} + x^{2k+1})}
        \big( \smb{1} + \smb{u (1+x^n)} \big) \\
      &= \smb{1 + x^n} \big( \smb{1} + \smb{u (1+x^n)} \big) \\
      &= \smb{1 + x^n} + \smb{u}
  \end{align*}
  This is exactly the relation we sought.
\end{proof}

\begin{prop}\label{prop[4].comparison_sq_trunc}
  Let $R \defiso \ff{2}\llbracket x \rrbracket$
  and $R_n \defiso \ff{2}[x]/(x^n)$ for $n \geq 1$.
  The canonical comparison morphism
  \[
    R^\times / R^{\times 2}
      \rightarrow \lim R_n^\times / R_n^{\times 2}
  \]
  is an isomorphism.
\end{prop}

\begin{proof}
  The surjectivity immediately follows from the fact that
  the induced maps $R_{n+1}^{\times 2} \rightarrow R_n^{\times 2}$
  are surjective.
  If $\smb{u}$ is in the kernel, then the image of $u$ under each truncation
  $R \rightarrow R_n$ is a square and in particular even. This implies in
  turn that $u$ only has even coefficients and is therefore a square in $R$.
  The comparison morphism is thus injective.
\end{proof}

\begin{rmk}\label{rmk[4].sq_power_series}
  The identification
  \[
    {\ff{2}\llbracket x \rrbracket}^\times / \,
      {\ff{2}\llbracket x \rrbracket}^{\times 2}
      \iso (\cyc{2})^\N
  \]
  resulting from \cref{prop[4].basis_sq_trunc} and
  \cref{prop[4].comparison_sq_trunc} is explicitly
  given by the formula
  \[
    a_\bullet \mapsto
      \smb[\Big]{\, \prod_{n=0}^\infty (1 + x^{2n + 1})^{a_n} }
  \]
  Indeed, it suffices to check that this formula holds at the level
  of the rings $\ff{2}[x] / (x^n)$.
\end{rmk}

\begin{cor}\label{cor[4].comparison_trunc}
  The canonical comparison morphism
  \[
    \GW^\sym(\ff{2}\llbracket x \rrbracket)
      \rightarrow \lim_n \GW^\sym\!\big(\ff{2}[x] / (x^n) \big)
  \]
  is an isomorphism.
\end{cor}

\begin{cor}\label{cor[4].calcul_power_series}
  The sequences $\smb{1}, \pfister{1 + x^{2\bullet + 1}}$
  and $\smb{1}, \pfister{x}, \pfister{1 + x^{2\bullet + 1}}$
  respectively induce canonical ring identifications
  \[
    \GW^\sym(\ff{2}\llbracket x \rrbracket) \iso \Z \oplus (\cyc{2})^\N
    \quad\text{and}\quad
    \GW^\sym(\ff{2}(\!( x )\!))
      \iso \Z \oplus \big(\, \cyc{2} \oplus (\cyc{2})^\N \,\big)
  \]
  of square zero extensions.
\end{cor}

\begin{proof}
  Let $R \defiso \ff{2}\llbracket x \rrbracket$ and
  $F \defiso \ff{2}(\!( x )\!)$, so that we have
  a short exact sequence
  \begin{center}
    \begin{tikzcd}[sep=.7cm]
      0
        \arrow[r] &
      \cyc{2}
        \arrow[r, "\smb{x}"] &
      F^\times / F^{\times 2}
        \arrow[r] &
      R^\times / R^{\times 2}
        \arrow[r] &
      0
    \end{tikzcd}
  \end{center}
  split by the map $R^\times / R^{\times 2} \rightarrow F^\times / F^{\times 2}$
  induced by the inclusion.
  We conclude using \cref{rmk[4].pfister_sqz,rmk[4].sq_power_series}.
\end{proof}


\end{document}